\documentclass[11pt]{amsart}

\usepackage[utf8x]{inputenc}
\usepackage[english]{babel}
\usepackage[all,cmtip]{xy}

\usepackage[pdftex]{graphicx}
\usepackage[margin=2.8cm]{geometry}
\usepackage{amsfonts}
\usepackage{amsmath}
\usepackage{amssymb}
\usepackage{amsthm}
\usepackage{dsfont}
\usepackage{tikz}
\usepackage{pgfplots}
\usepackage{MnSymbol}

\usepackage[T1]{fontenc}

\usepackage{paralist}

\usepackage[pdftex]{hyperref}

\pgfplotsset{compat=1.9}

\renewcommand\epsilon{\varepsilon}
\renewcommand\emptyset{\varnothing}
\renewcommand\phi{\varphi}
\newcommand\N{\mathbb{N}}

\newcommand\R{\mathbb{R}}

\newtheorem{thm}{Theorem}[section]
\newtheorem*{thm*}{Theorem}

\newtheorem{lem}[thm]{Lemma}

\newtheorem{quest}{Question}

\theoremstyle{definition}

\newtheorem{rem}[thm]{Remark}

\title[On embeddings of finite subsets of $\ell_p$]{On embeddings of finite subsets of $\ell_p$}

\author{James Kilbane}
\address{Department of Pure Maths and Mathematical Statistics, University of Cambridge}
\email{jk511@cam.ac.uk}

\date{\today}

\begin{document}
\begin{abstract}
We study finite subsets of $\ell_p$ and show that, up to nowhere dense and Haar null complement, all of them embed isometrically into any Banach space that uniformly contains $\ell_p^n$.
\end{abstract}

\maketitle

\section{Introduction}

Our starting point is the following question due to Ostrovskii\cite{ostrov}:
\begin{quest}\label{quest:infdim}
Suppose $1 < p < \infty$, and that $X$ is a Banach space that contains an isomorphic copy of $\ell_p$. Then does any finite subset of $\ell_p$ embed isometrically into $X$?
\end{quest}

A consequence of Krivine's theorem, which we recall in Section 2, is that any Banach space containing $\ell_p$ isomorphically, contains the spaces $\ell_p^n$, $n \in \N$, almost isometrically. The above question is asking for a natural strengthening of this fact.

The following partial result for Question \ref{quest:infdim} in the case $p=2$ was proved by Shkarin in \cite{shkarin}:
\begin{thm}[Lemma 3 of \cite{shkarin}]\label{thm:l2}
Suppose $X$ is any infinite-dimensional Banach space and that $Z$ is any affinely independent subset of $\ell_2$. Then $Z$ embeds isometrically into $X$.
\end{thm}

A different proof was given in \cite{me}, and the methods of both \cite{me} and \cite{shkarin} inspired the proofs in this article. We note that, by Dvoretzky's theorem, $\ell_2^n$ almost isometrically embeds into $X$ for any infinite-dimensional Banach space $X$. Thus, in the case $p=2$, Theorem \ref{thm:l2} provides a partial positive answer to the following variant of Question \ref{quest:infdim}:
\begin{quest}\label{quest:findim}
Suppose that $1 < p < \infty$ and that $X$ is a Banach space uniformly containing the spaces $\ell_p^n$, $n \in \N$. Then does any finite subset of $\ell_p$ embed isometrically into $X$?
\end{quest}
As before, the weaker conclusion that finite subsets of $\ell_p$ embed almost isometrically into such a space $X$ follows from Krivine's theorem, or more precisely, a finite quantitative version of it (see Theorem \ref{thm:krivine} below.)

There are natural analogues of Questions \ref{quest:findim} and \ref{quest:infdim} for $p= \infty$. Since any $n$-point metric space embeds isometrically into $\ell_\infty^n$, the conclusion in any such analogue is that $X$ contains isometrically all finite metric spaces. The assumption on $X$ is one of the following (in decreasing order of strength): $X$ contains an isomorphic copy of $\ell_\infty$; $X$ contains an isomorphic copy of $c_0$; $X$ contains the spaces $\ell_\infty^n$, $n \in \N$, uniformly. The answer for each of these questions is, however, negative. Indeed, let $X$ be a strictly convex renorming of $\ell_\infty$. Then subsets of $X$ have the unique metric midpoint property, ie, there is no collection of 4 distinct	 points $x,y,z,w \in X$ such that $d(x,z) = d(z,y) = d(x,w) = d(w,y) = \frac{1}{2} d(x,y)$. However, there are finite metric spaces with this property, and thus such a metric space does not embed isometrically into $X$. In \cite{me} we showed a positive result similar to Theorem \ref{thm:l2}. Let us call a metric space \emph{concave} if it contains no three distinct points $x,y,z$ such that $d(x,z) = d(x,y) + d(y,z)$. Then,

\begin{thm}[Theorem 4.3 of \cite{me}]\label{thm:linf}
Suppose that $X$ is some infinite-dimensional Banach space such that the spaces $\ell_\infty^n$, $n \in \N$, uniformly embed into $X$. Then if $Z$ is any finite concave metric space, $Z$ embeds isometrically into $X$.
\end{thm}

In this paper we obtain a partial positive answer to Question \ref{quest:findim} similar to Theorems \ref{thm:l2} and \ref{thm:linf}. As in the case of $p=2$ there remains a class of subsets of $\ell_p$ that our proof does not handle.This collection is certainly small in a strong sense. Our main theorem is as follows,

\begin{thm}\label{thm:main}
Suppose $1 < p < \infty$ and that $Z$ is a Banach space that uniformly contains the spaces $\ell_p^n$, $n \in \N$. Then, for each $n \in \N$, the set of $n$-point subsets of $\ell_p$ that do not embed isometrically into $Z$ is nowhere dense and Haar null.
\end{thm}

We now describe how our paper is organized. We shall also explain why the case of $p \in (1,\infty)$ is more difficult than the special cases of $p=2$ and $p=\infty$ and how we handle the additional difficulty.

In Section \ref{sec:def} we recall various definitions and results that will be used throughout the article (and have already been used in this introduction). The proof of Theorem \ref{thm:main} begins in Section \ref{sec:perturb}. Here we prove a result (see Theorem \ref{thm:perturb}) that may be of independent interest: almost all $n$-point subsets of $\ell_p^n$ have the property that small perturbations of that subset remain subsets of $\ell_p^n$. In Section \ref{sec:brou} we introduce \emph{Property K} of finite subsets of $\ell_p$. Our aim will be to show that every finite subset of $\ell_p$ with Property $K$ embeds isometrically into a Banach space $X$ that satisfies the assumption of Theorem \ref{thm:main}.

For general $p \in (1,\infty)$, Property $K$ plays the r\^{o}le of affine independence in the case $p=2$, or concavity in the case $p=\infty$. For $p=2$, any $n$-point subset of $\ell_2$ embeds isometrically into $\ell_2^n$ via an orthogonal transformation which preserves affine independence. For $p=\infty$, any $n$-point metric space embeds into $\ell_\infty^n$ via an isometry, which preserves concavity. For general $p \in (1,\infty)$ it is not even clear if a finite subset of $\ell_p$ embeds isometrically into $\ell_p^N$ for any $N$. In fact, this is true: Ball proved in \cite{ball} that any $n$-point subset of $\ell_p$ embeds isometrically into $\ell_p^N$ with $N= \binom{n}{2}$. The difficulty is that Ball's proof is not constructive, and Property $K$ is somewhat technical. In Section \ref{sec:brou} we prove a version of Ball's result (Lemma \ref{lem:ball}) which is much weaker, in the sense that $N$ will depend on the subset. However, Lemma \ref{lem:ball}, will show that our embedding will preserve Property $K$. 

\begin{rem}
In this article, we do not pay much attention to the case $p=1$. Indeed, as stated, Question \ref{quest:infdim} is false. As for $\ell_\infty$, there is a strictly convex renorming $X$ of $\ell_1$, and no finite subset of $\ell_1$ that fails the unique metric midpoint property embeds isometrically into such an $X$. However, one might expect a result similar to Theorem \ref{thm:linf} to hold, when there's a restriction on the type of subset we consider. The methods of this paper rely heavily on the differentiability of the norm of $\ell_p$ for $1<p<\infty$ which fails for $p=1$. Thus our techniques only produce weak conclusions in the case $p=1$.
\end{rem}
\section{Classical Results and Notation}\label{sec:def}
\subsection{Banach Space Definitions and Classical Results:}\label{sec:ban} Throughout this paper, for simplicity, we will only be interested in \emph{real} Banach spaces.

Suppose that $X$ and $Y$ are Banach spaces. The \emph{Banach-Mazur distance between $X$ and $Y$} is defined by $d(X,Y) = \inf \{ \|T\| \|T^{-1}\| : T$ is an isomorphism from $X$ to $Y\}$. We say that a Banach space $X$ is \emph{$C$-isomorphic} to a Banach space $Y$ if there is a linear isomorphism $T:X \rightarrow Y$ such that $\|T\| \|T^{-1}\| \leq C$. We say that a Banach space $X$ \emph{almost isometrically contains} a Banach space $Y$, or that $Y$ \emph{almost isometrically embeds into $X$}, if for each $\epsilon > 0$ there is a subspace $Z$ of $X$ such that $Z$ is $(1+\epsilon)-$isomorphic to $Y$. We say that a Banach space $X$ \emph{uniformly contains} spaces $X_n$, $n \in \N$, if there exist a constant $C$ and subspaces $Y_n$ of $X$ such that $Y_n$ is $C$-isomorphic to $X_n$ for all $n$.

We will need the following quantitative version of Krivine's theorem:

\begin{thm}\label{thm:krivine}
Let $1 \leq p \leq \infty$, $C \geq 1$, $\epsilon > 0$ and $k \in \N$. Then there is some $n$ (dependent on $p,C,k$ and $\epsilon$) such that if a Banach space $X$ is $C$-isomorphic to $\ell_p^n$ then there is a subspace of $X$ that is $(1+\epsilon)$-isomorphic to $\ell_p^k$.
\end{thm}

For a proof of this theorem, including estimates of the constants involved, we refer the reader to \cite{kriv}.

We introduce a notion of a null set in infinite-dimensional Banach spaces. A well known fact is that if $X$ is infinite-dimensional and separable, and $\mu$ is a translation-invariant Borel measure on $X$, then $\mu$ either assigns $0$ or $\infty$ to every open subset of $X$. However, there are several useful notions of null set in Banach spaces under which the null sets form a translation-invariant $\sigma$-ideal. One such notion, that we shall use, is that of a \emph{Haar null} set. A Borel set $A \subset X$ is called \emph{Haar null} if there is a Borel probability measure $\mu$ on $X$ such that $\mu(x+A) = 0$ for every $x \in X$. It is easy to see that if for some $n \in \N$ there is an $n$-dimensional subspace $Y$ of $X$ such that the measure $\lambda_n(Y \cap (A+x)) = 0$ for all $x \in X$, where $\lambda_n$ is $n$-dimensional Lebesgue measure, then $X$ is Haar null. More on sets of this type, and on other notions of nullity, can be found in \cite[Chapter 6]{gnfa}.

\subsection{Submersions:} We will need a fact from the theory of Differential Geometry related to submersions. Suppose we have a $C^1$-map $\Phi: \mathbb{R}^n \rightarrow \mathbb{R}^m$ where $n \geq m$. We say that $\Phi$ is a \emph{submersion} at a point $x$ if the derivative $D\Phi|_x$ of $\Phi$ at $x$ has rank $m$. The following result is known as the Submersion Theorem and can be found in any introductory text on Differential Geometry:
\begin{thm}
Suppose $\Phi:\R^n \rightarrow \R^m$ is a $C^1$-map, where $n \geq m$. If $\Phi$ is a submersion at a point $x$, then there are open sets $A \subset \R^n$ and $B \subset \R^m$ with $x \in A$, $\Phi(x) \in B$ and $\Phi(A) = B$. Moreover, there is a $C^1$-map $\Psi:B \rightarrow A$ such that $\Phi \circ \Psi$ is the identity on $B$ and $\Psi(\Phi(x)) = x$.
\end{thm}
\section{A Theorem about Finite Subsets of \texorpdfstring{$\ell_p^n$}{lpn}}\label{sec:perturb}
In this section we establish a preliminary result that may be of independent interest. Suppose that $Z$ is a metric space on a sequence of points $(z_i)_{i=1}^n$ and $Y$ is a metric space on a sequence of points $(y_i)_{i=1}^n$. We say that $Y$ is an \emph{$\epsilon$-perturbation} of $Z$ if for each pair $i,j$ we have that $|d_Z(z_i,z_j) - d_Y(y_i,y_j)| <\epsilon$.

In finite dimensions, the phrase \emph{almost all} will only be used with respect to Lebesgue measure. Throughout this section, we fix some $n \in \N$ and $p \in \R$ with $1 < p < \infty$. We denote by $\|.\|$ the $p$-norm on $\ell_p^n$.

\begin{thm}\label{thm:perturb}
For almost all $n$-point subsets $X$ of $\ell_p^n$, there is an $\epsilon > 0$ such that if $Y$ is an $\epsilon$-perturbation of $X$ then $Y$ isometrically embeds into $\ell_p^n$.
\end{thm}
For our purposes, we will need a slightly stronger property of an $n$-point subset $X$ of $\ell_p^n$. We will need that an $\epsilon$-perturbation of $X$ isometrically embeds into $\ell_p^n$ in a way that depends continuously on the perturbation (in a way we will make precise in the sequel.) This is the content of Theorem \ref{thm:precise} below, from which Theorem \ref{thm:perturb} will easily follow. To state Theorem \ref{thm:precise} we will first develop some notation.

Let $M = M_n = \underbrace{\R^n \times \dots \times \R^n}_{n \text{ times}}$ and let $U = U_n$ denote the $n \times n$ upper triangular matrices with 0 on the diagonal. We let $e_1,\dots,e_n$ be the standard basis of $\R^n$ and $e_i^j$ be the element of $M$ with $e_j$ in the $i$\textsuperscript{th} co-ordinate and zero everywhere else. Note that $e_i^j$, $1 \leq i,j \leq n$, form a basis of $M$. Given $x=(x_1,\dots,x_n) \in M$ we denote the $j$\textsuperscript{th} co-ordinate (with respect to the standard basis) of the vector $x_i$ as $x_i^j$ so that $x = \sum_{i,j} x_i^j e_i^j$. Let $E_{ij}$ be the $n\times n$ matrix with $1$ in the $(i,j)$-entry and 0 elsewhere. Note that $E_{ij}$, $1 \leq i < j \leq n$ forms a basis for $U$, so the dimension of $U$ is $\binom{n}{2}$.

We define the map $F = F_n : M \rightarrow U$ by $$F(x_1,\dots,x_n) = (\|x_i-x_j\|^p)_{1 \leq i < j \leq n}. $$We observe that $F$ is a $C^1$-map. Indeed, by computing the partial derivatives in the direction $e_l^k$ we get:
\begin{equation}\label{eq:derivcomp}
\frac{\partial F}{\partial e_l^k} (z_1,\dots,z_n) = \left( p|z^k_i - z^k_j |^{p-1} \text{ sgn}(z_i^k - z_j^k) (\delta_{il} - \delta_{jl}) \right)_{1 \leq i < j \leq n},
\end{equation}
and these are evidently continuous. Theorem \ref{thm:perturb} says that $F$ is locally open at almost all $n$-tuples $(x_1,\dots,x_n)$. This is contained in the following theorem:
\begin{thm}\label{thm:precise}
Let $F: M \rightarrow U$ be defined as above. Set $G = G_n = \{x \in M : DF|_x \text{ has rank } \binom{n}{2} \}$. Then $G$ is an open subset of $M$ whose complement has measure zero (and is thus nowhere dense.) Moreover, given $x \in G$, there is an open subset $A$ of $M$ containing $x$, an open subset $B$ of $U$ containing $F(x)$ and a $C^1$-map $\Phi: B \rightarrow A$ such that $F \circ \Phi = $ Id$_B$ and $\Phi(F(x)) = x$.
\end{thm}

Let us briefly spell out how Theorem \ref{thm:perturb} follows from Theorem \ref{thm:precise}. Suppose that $\{x_1,\dots,x_n\}$ is an $n$-point subset of $\ell_p^n$ and that $x=(x_1,\dots,x_n) \in G$. Define $X_{ij} = \|x_i-x_j\|^p$ and $X = \left(X_{ij}\right)_{1 \leq i < j \leq n}$. Then, since $x \in G$, by Theorem \ref{thm:precise} there are open subsets $A$ of $M$ and $B$ of $U$ such that $x \in A$, $F(x) = X \in B$ and $F(A) = B$. Thus there is some $\epsilon > 0$ such that if $|Y_{ij} - X_{ij}| < \epsilon$ for all $i,j$, then $(Y_{ij})_{1 \leq i < j \leq n}$ is an element of $B$ and thus is the image under $F$ of some $y=(y_1,\dots,y_n) \in A$. Hence $Y$ defines a metric on an $n$-point set and the resulting metric space embeds isometrically into $\ell_p^n$. This is slightly more than the statement that $\epsilon$-perturbations of the metric space $\{x_1,\dots,x_n\}$ with the inherited metric embed isometrically into $\ell_p^n$.

\begin{proof}[Proof of Theorem \ref{thm:precise}]
We first show that $G$ is open. Indeed, if $x \in G$, then there is a linear map $B: U \rightarrow M$ such that $DF|_x \circ B = $ Id$_U$. Since $DF$ is continuous, there is some $\epsilon > 0$ such that whenever $y$ is such that $\|x-y\| < \epsilon$, $\|DF|_y \circ B - $ Id$_U \| < 1$. Thus, $DF|_y \circ B$ is invertible, and $DF|_y$ has full rank.

We now show that $M \setminus G$ has measure zero. Once we do this, the proof of the theorem is then complete. Indeed, the rest of the statement of Theorem \ref{thm:precise} follows immediately from the Submersion Theorem.

The proof that $M \setminus G$ has measure zero is done in several steps. We first identify a certain subset of $G$.
\begin{lem}\label{lem:tri}
Let $H = \{(x_1,\dots,x_n)\in M : x_i = e_i + \sum_{j=i+1}^n x_i^j e_j$ for each $i=1,\dots,n\}$. Then if $x \in H$, the partial derivatives $\frac{\partial F}{\partial e_l^k} (x)$, $1 \leq k < l \leq n$ are linearly independent. In particular, $H \subset G$.
\end{lem}
\begin{proof}
Fix $x = (x_1,\dots,x_n) \in H$. By \eqref{eq:derivcomp} we see that the $(i,j)$-entry of $\frac{\partial F}{\partial e_l^k} (x)$ is zero unless $j=l$ and $i \leq k$. We can hence expand $\frac{\partial F}{\partial e_l^k} (x)$ in terms of the matrices $E_{kl}$ as follows, $$ \frac{\partial F}{\partial e_l^k} (x) = - p E_{kl} + \sum_{i=1}^{k-1} \alpha_i^k E_{il},$$where $\alpha_i^k$ are constants depending on $x$. It follows by induction on $k$ that $E_{kl}$ is in the span of $\frac{\partial F}{\partial e_i^j} (x)$ for all $1 \leq k < l \leq n$. This completes the proof of the lemma.
\end{proof}

Let us now define $V = \{x = (x_1,\dots,x_n) \in M : $ there are $i,j,k \in \{1,\dots,n\}$ such that $i \neq j$ and $x_i^k = x_j^k \}$. Note that $M \setminus V$ has finitely many connected components which are open and convex. Since $\mu(V) = 0$, in order to show that $\mu(M \setminus G) = 0$, it suffices to show that $\mu(C \setminus G) = 0$ for every connected component $C$ of $M \setminus V$. The following lemma will be vital to this aim.

\begin{lem}\label{lem:line}
Suppose that $x = (x_1,\dots,x_n)$ and $y = (y_1,\dots, y_n)$ are two points in the same connected component of $M \setminus V$, and suppose that $\frac{\partial F}{\partial e_i^j} (x)$, $1 \leq i < j \leq n$, are linearly independent. Then, for all but finitely many values of $t \in [0,1]$, the partial derivatives $\frac{\partial F}{\partial e_i^j} ((1-t) x + ty)$, $1 \leq i < j \leq n$, are linearly independent. In particular, for all but finitely many values of $t \in [0,1]$, we have that $(1-t) x + ty \in G$.
\end{lem}

\begin{proof}
Define $J$ to be the set $\{(k,l) : 1 \leq k < l \leq n\}$. For $\sigma = (i,j) \in J$ we will write $e_\sigma = e_i^j$, and for $X \in U$ we will write $X_\sigma$ for the $(i,j)$-entry of $X$. By assumption, the $J \times J$ matrix given by $\left( \left( \dfrac{\partial F}{\partial e_\sigma}(x)\right)_\rho \right)$ has non-zero determinant. We now define a function $g:[0,1] \rightarrow \R$ by setting $$g(t) = \det \left( \left( \dfrac{\partial F}{\partial e_\sigma} ((1-t)x + ty) \right)_\rho \right) = \det(X(t)). $$Using \eqref{eq:derivcomp} and the fact that $x$ and $y$ are from the same component of $M \setminus V$, for each $\sigma,\rho \in J$, the matrix $X(t)$ has $(\sigma,\rho)$-entry $p(a_{\sigma,\rho} t + b_{\sigma,\rho})^{p-1} \epsilon_{\sigma,\rho}$ where $a_{\sigma,\rho}$ and $b_{\sigma,\rho}$ are non-zero constants with $a_{\sigma,\rho}t + b_{\sigma,\rho} > 0$ for all $t \in [0,1]$ and $\epsilon_{\sigma,\rho} \in \{-1,0,1\}$. 

By compactness there is an open connected subset $U$ of $\mathbb{C}$ containing $[0,1]$ such that the real part of $a_{\sigma,\rho} t + b_{\sigma,\rho}$ is positive for each $t \in U$. It follows that the function $g$ extends analytically to all of $U$, and therefore by the identity principle (and the fact that $g(0)$ is non-zero), $g$ has at most finitely many zeroes in $[0,1]$. 
\end{proof}

Consider the subset $R$ of $M$ defined by $$R = \{(x_1,\dots,x_n) \in M : x_i^i > x_j^i \text{ for each } 1 \leq i < j \leq n\}. $$ Note that for each component $C$ of $M \setminus V$ either $C \subset R$ or $C \cap R = \emptyset$. We next show that in order to prove that $\mu(C \setminus G) = 0$ for every component $C$ of $M \setminus V$, it is sufficient to consider components $C$ such that $C \subset R$.

Fix $(x_1,\dots,x_n) \in M \setminus V$. Define a permutation $\pi \in S_n$ recursively as follows: for $j=1,\dots,n$, let $\pi(j)$ be the unique $i \in \{1,\dots,n\} \setminus \{\pi(1), \dots, \pi(j-1)\}$ such that $$x_i^j > x_k^j \text{ for all } k \in \{1,\dots,n\} \setminus \{\pi(1),\dots,\pi(j-1), i\}. $$ It then follows that $x^j_{\pi(j)} > x^j_{\pi(k)}$ for all $1 \leq j < k \leq n$, and hence $(x_{\pi(1)},\dots,x_{\pi(n)}) \in R$. 

Define a map $A_\pi : M \rightarrow M$ by $A_\pi(y_1,\dots,y_n) = (y_{\pi(1)},\dots,y_{\pi(n)})$, and a map $B_\pi : U \rightarrow U$ by $B_\pi((X_{ij})_{1\leq i < j \leq n}) = (Y_{ij})_{1\leq i < j \leq n}$ where \begin{equation*}
Y_{ij} = \begin{cases}
	X_{\pi(i),\pi(j)} & \text{if }\pi(i) < \pi(j) , \\
	X_{\pi(j), \pi(i)} & \text{if } \pi(j) < \pi(i) .
\end{cases}
\end{equation*} 

We note that $B_\pi^{-1} F A_\pi = F$, and thus $B_\pi^{-1} DF|_{A_{\pi}(x)} A_\pi = DF|_x$, so to verify that $F$ has full rank at $x$, it is sufficient to verify that $F$ has full rank at $A_\pi(x)$, which lies in $R$. This completes the proof that it is sufficient to show that $\mu(C \setminus G) = 0$ whenever $C$ is a component of $M \setminus V$ with $C \subset R$.

Fix a component $C$ of $M \setminus V$ with $C \subset R$. If $\mu(C \setminus G) > 0$, then by Lebesgue's density theorem, there is a point $y \in C$ such that $\lim_{\epsilon \rightarrow 0} \frac{ \mu(B_{\epsilon)}(y) \cap (C \setminus G))}{\mu(B_\epsilon(y))} = 1$. For $i,j \in \{1,\dots,n\}$, define $$x_i^j = \begin{cases} 1 & \text{if } i=j, \\ 1 & \text{if } y_i^j > y_j^j, \\ 0 & \text{else.} \end{cases}$$ It is easy to verify that if $y_i^k < y_j^k$ then $x_i^k \leq x_j^k$, and thus $(1-t) x + ty \in C$ for all $t \in (0,1]$. Moreover, since $y \in R$, we have $x \in H$. It follows by Lemma \ref{lem:tri} that the partial derivatives $\frac{\partial F}{\partial e_i^j}(x)$, $1 \leq i < j \leq n$, are linearly independent. Hence there is an $\epsilon > 0$ such that at each $z \in B_\epsilon(x)$ the same holds, ie, $\frac{\partial F}{\partial e_i^j}(z)$, $1 \leq i < j \leq n$, are linearly independent. Choose  $t \in (0,1)$ such that $z = (1-t) x + ty \in B_\epsilon(x)$. Then $z \in B_\epsilon(x) \cap C$, so there is some $\delta > 0$ such that $B_\delta(z) \subset B_\epsilon \cap C$.

The Lebesgue density at $y$ is equal to $1$, so by making $\delta$ smaller, we may assume that $B_\delta(y) \subset C$ and $\mu(B_\delta(y) \setminus G) > 0$. By Lemma \ref{lem:line}, each line in the direction $y-x$ through a point in $B_\delta(z)$ intersects $B_{\delta}(y) \setminus G$ in at most finitely many points.  The lines in the direction $y-x$ through $B_\delta(z)$ can be parametrised by where they intersect the hyperplane through $z$ whose normal is $y-x$. This is a $(\binom{n}{2} - 1)$-dimensional hyperplane. The measure of $B_\delta(y) \setminus G$ can be given, by Fubini's theorem, as $$ \mu(B_\delta(y) \setminus G) = \int_{\R^{\binom{n}{2} - 1}} \int_{[a_s,b_s]} 1_{L(s) \cap B_\delta(y) \setminus G} d\mu^\prime ds $$where $L(s)$ is the line through the point $s$ in the previously mentioned hyperplane going through $s$, $[a_s,b_s]$ is the interval for which the line $L(s)$ intersects the sphere $B_\delta(y)$ and $\mu^\prime$ is 1-dimensional Lebesgue measure. This integral is equal to zero, as $L(s) \cap B_\delta(y) \setminus G$ is finite. This is a contradiction on $y$ being a point of Lebesgue density, and thus of $C \setminus G$ having non-zero measure. Thus $\mu(C \setminus G) = 0$ and the proof of Theorem \ref{thm:precise} is complete.
\end{proof}

\section{The Proof of Theorem \ref{thm:main}}\label{sec:brou}
Given a subset $M = \{m_1,\dots,m_n\}$ of $\N$ with $m_1 < m_2 < \dots < m_n$, if $x = (x_i)_{i=1}^\infty \in \ell_p$ or $x=(x_i)_{i=1}^N \in \ell_p^N$ with $N \geq m_n$, we define $P_M(x) = (x_{m_1},\dots,x_{m_n})$. If $N \in \N$, we write $P_N$ instead of $P_{\{1,\dots,N\}}$.

We say that an $n$-tuple $(x_1,\dots,x_n)$ in $\ell_p$ (or $\ell_p^N$) has \emph{Property K} if there is an $M \subset \N$ (or $M \subset \{1,\dots,N\}$ respectively) of size $n$ such that $(P_M x_1,\dots, P_M x_n) \in G_n$, where $G_n$ is the set defined in Theorem \ref{thm:precise}. Note that the set of $n$-tuples with property $K$ is open since the set $G_n$ is open.

We prove Theorem \ref{thm:main} by showing that the closed set of $n$-tuples without Property $K$ is Haar null (and thus nowhere dense), and that an $n$-tuple with Property $K$ embeds isometrically into a Banach space that satisfies the assumption of Theorem \ref{thm:main}. We will need three lemmas.

\begin{lem}\label{lem:ball}
Suppose that $x = (x_1,\dots,x_n)$ is an $n$-tuple in $\ell_p$ with Property $K$. Then there is some $N \in \N$, and vectors $y_1,\dots, y_n \in \ell_p^N$ such that $\|y_i-y_j\| = \|x_i-x_j\|$ and the $n$-tuple $(y_1,\dots,y_n)$ has Property $K$.
\end{lem}
\begin{rem}
This is the variant of Ball's result mentioned in the Introduction. Here $\|.\|$ denotes the $\ell_p$ norm.
\end{rem}
\begin{proof}[Proof of Lemma \ref{lem:ball}]
Let $M \subset \N$ be such that $|M| = n$ and $(P_M x_1, \dots, P_M x_n) \in G_n$. After an isometry (permuting the indices), we may assume without loss of generality that $M = \{1,\dots, n\}$. Then, since $(P_n x_1,\dots,P_n x_n) \in G_n$ and $G_n$ is open, there is some $\epsilon > 0$ such that if $z_i \in \ell_p^n$ and $\|z_i - P_n x_i\| < \epsilon$ then $(z_1,\dots,z_n) \in G_n$.

Since $(P_n x_1,\dots, P_n x_n) \in G_n$, by Theorem 3.2, there are open sets $A \ni (P_n x_1,\dots, P_n x_n)$, $B \ni F(P_n x_1,\dots, P_n x_n)$ and a $C^1$-map $\Phi:B \rightarrow A$ such that $F \circ \Phi = $ Id$_B$ and $\Phi(F(x)) = x$.

Fix $N \geq n$, and define $\rho_{ij} = \rho_{ij}(N)$ by $\|x_i - x_j\|^p = \|P_N x_i - P_N x_j\|^p + \rho_{ij}$. Since $\rho_{ij} \rightarrow 0$ as $N \rightarrow \infty$, there is an $N>n$ such that the element $Z = Z(N) = (\|P_n x_i - P_n x_j\|^p + \rho_{ij})_{1 \leq i < j \leq n}$ of $U$ is in the set $B$. Set $z = z(N) = (z_1,\dots,z_n) = \Phi(Z)$. By the continuity of $\Phi$ at the point $F(P_n x_1,\dots, P_n x_n)$, if $N$ is sufficiently large, then $\|z_i - P_n x_i\| < \epsilon$, and hence $(z_1,\dots, z_n) \in G_n$.

We now define the points $y_1,\dots, y_n \in \ell_p^N$ by:
\begin{itemize}
	\item $P_n y_i = z_i$
	\item $(P_N - P_n) y_i = (P_N - P_n) x_i$.
\end{itemize}
We now verify that $(y_1,\dots,y_n)$ has Property $K$, and that $\|y_i - y_j\| = \|x_i - x_j\|$. The first of these is clear, $(P_n y_1,\dots,P_n y_n)$ is in $G_n$ by construction. 

To verify that $\|y_i - y_j\| = \|x_i - x_j\|$, note that $$\|y_i - y_j\|^p = \|P_n y_i - P_n y_j\|^p + \|(P_N - P_n) y_i - (P_N - P_n) y_j\|^p,$$which is equal to$$\|z_i - z_j\|^p + \|(P_N - P_n) x_i - (P_N - P_n) x_j\|^p.$$ By the definition of $(z_1,\dots,z_n)$, we see that $\|z_i - z_j\|^p = \|P_n x_i - P_n x_j\|^p + \rho_{ij}$. By the definition of $\rho_{ij}$, we thus get that $\|y_i-y_j\|^p = \|x_i-x_j\|^p$.
\end{proof}

We have now shown that if a subset of $\ell_p$ has Property $K$, then it is isometric to a subset of $\ell_p^N$ with Property $K$. We next show a slight variant of Theorem \ref{thm:precise}.

\begin{lem} \label{lem:perturb}
Suppose $x = (x_1,\dots,x_n)$ is an $n$-tuple in $\ell_p^N$, $N \geq n$, with Property $K$. Then there is some $\epsilon > 0$ such that any $\epsilon$-perturbation of $X$ can be embedded into $\ell_p^N$ with the embedding depending continuously on the perturbation.
\end{lem}
At the beginning of the proof of Lemma \ref{lem:perturb} we will make it clear what continuous dependence on the perturbation means in a way similar to the precise statement of Theorem \ref{thm:precise}.

\begin{proof}
Define $\tilde{F} : \underbrace{\R^N \times \dots \times \R^N}_{n \text{ times}} \rightarrow U_n$ by $$\tilde{F}(y_1,\dots,y_n) = \left( \|y_i-y_j\| \right)_{1 \leq i < j \leq n},$$where we note that there is no $p$\textsuperscript{th} power of the norm. Our goal is to show that there is an open subset $\tilde{B}$ of $U_n$ and a continuous map $\Psi : \tilde{B} \rightarrow \underbrace{\R^N \times \dots \times \R^N}_{n \text{ times}}$ such that:
\begin{itemize}
	\item $\tilde{F}(x) \in \tilde{B}$
	\item $\Psi(\tilde{F}(x)) = x$
	\item $\tilde{F} \circ \Psi = $ Id$_{\tilde{B}}$.
\end{itemize}

Let $M \subset \{1,\dots,N\}$ be such that $|M| = n$ and $(P_M x_1,\dots, P_M x_n) \in G_n$. Again, without loss of generality, we may assume that $M = \{1,\dots,n\}$.

By Theorem \ref{thm:precise}, there exist open sets $A \ni (P_n x_1,\dots, P_n x_n)$, $B \ni F(P_n x_1,\dots, P_n x_n)$ and a $C^1$-map $\Phi: B \rightarrow A$ such that $\Phi(F(P_n x_1,\dots, P_n x_n)) = (P_n x_1,\dots,P_n x_n)$ and $F \circ \Phi = $ Id$_B$. Fix $\epsilon > 0$ such that if $Y = (Y_{ij})_{1 \leq i < j \leq n}$ is such that $|Y_{ij} - \|x_i-x_j\|^p | < \epsilon$, then $Y \in B$.

Choose $\delta = \delta(\epsilon) > 0$ to be specified later. We set $\tilde{B} = \{Y \in U_n : |Y_{ij} - \|x_i - x_j\| | < \delta$ for all pairs $i,j\}$.

Fix $Y = (Y_{ij})_{1\leq i < j \leq n} \in \tilde{B}$. We define $\Psi(Y)$ similarly to the definition of the points $(y_1,\dots,y_n)$ in the proof of Lemma \ref{lem:ball}. Define $\rho_{ij} = Y_{ij} - \|x_i - x_j\|$ and $\epsilon_{ij} = \epsilon_{ij}(Y)$ by $(\|x_i-x_j\| + \rho_{ij})^p = \|x_i - x_j\|^p + \epsilon_{ij}$. If $|\rho_{ij}|$ is sufficiently small (ie, our choice of $\delta$ is sufficiently small), then $(\|P_n x_i - P_n x_j\|^p + \epsilon_{ij})_{1 \leq i < j \leq n}$ is in $B$. Define $z_i = \Phi((\|P_n x_i - P_n x_j\|^p + \epsilon_{ij})_{1 \leq i < j \leq n})$. We then set $\Psi(Y)$ to be the $n$-tuple $(y_1,\dots,y_n)$ where:
\begin{itemize}
	\item $P_n y_i = z_i$
	\item $(P_N - P_n) y_i = x_i$.
\end{itemize}

We verify that $\|y_i - y_j\| = \|x_i - x_j\| + \rho_{ij} = Y_{ij}$, ie that $\tilde{F}(\Psi(Y)) = Y$, as this is the only one of the three properties listed above that is non-trivial.

Indeed, $$\|y_i-y_j\|^p = \|P_n y_i - P_n y_j\|^p + \|(P_N - P_n) y_i - (P_N - P_n)y_j\|^p,$$which (by the definition of $y_i$) equals $$\|z_i-z_j\|^p + \|(P_N - P_n) x_i - (P_N - P_n)x_j\|^p$$and this is equal (by the definition of $z_i$) to $$\|x_i - x_j\|^p + \epsilon_{ij}.$$By the definition of $\epsilon_{ij}$, this is equal to $(\|x_i-x_j\| + \rho_{ij})^p$, which is as required.
\end{proof}

Our next lemma shows that if we have an $n$-point subset of $\ell_p^N$ with Property $K$, then it embeds isometrically into any Banach space satisfying the assumption of Theorem \ref{thm:main}. This result is, in some sense, dual to Theorem \ref{thm:perturb}. Where Theorem \ref{thm:perturb} says \emph{small perturbations of the metric space embed into the Banach space}, this is saying that \emph{the metric space embeds into small perturbations of the Banach space}.

\begin{lem}\label{lem:final}
Suppose $x=(x_1,\dots,x_n)$ is an $n$-tuple in $\ell_p^N$, $N \geq n$, with Property $K$. Then there is some $\delta > 0$ such that if $d(E,\ell_p^N) < 1+\delta$ then $\{x_1,\dots,x_n\}$ with the metric inherited from $\ell_p^N$ embeds isometrically into $E$.
\end{lem}
\begin{proof}
Let $\tilde{F}, \tilde{B}$ and $\Psi$ be as in the proof of Lemma \ref{lem:perturb}. Choose $\epsilon > 0$ such that if $Y = (Y_{ij})_{1 \leq i < j \leq n} \in U$, $|Y_{ij} - \|x_i-x_j\| | < \epsilon$, then $Y \in \tilde{B}$. Fix some $\delta > 0$ and let $E$ be an $N$-dimensional Banach space such that $d(E,\ell_p^N) < 1+\delta$. We will find the value of $\delta$ later, and it will be expressed in terms of $x$ and $\epsilon$ only. We may assume that $E = (\R^N,\|.\|_E)$ and that the norm on $E$ satisfies $\|y\|_E \leq \|y\| \leq (1+\delta) \|y\|_E$, where as usual $\|.\|$ denotes the $\ell_p$ norm.

Let $\rho = (\rho_{ij})_{1 \leq i < j \leq n}$ be an element of the space $[0,\epsilon]^{\binom{n}{2}}$. We define a metric space $Z(\rho)$ as follows:
\begin{itemize}
	\item $Z(\rho)$ is a metric space on $n$ distinct points $z_1,\dots,z_n$.
	\item $d(z_i,z_j) = \|x_i-x_j\| + \rho_{ij}$.
\end{itemize}
By the choice of $\epsilon$, and since $\tilde{F} \circ \Psi = $ Id$_{\tilde{B}}$, it follows that $Z(\rho)$ is a metric space isometric to a subset of $\ell_p^N$. Through slight abuse of notation, in what follows we identify $Z(\rho)$ with its distance matrix, ie, $Z(\rho) = (d(z_i,z_j))_{1 \leq i < j \leq n}$.

Now define $\varphi:[0,\epsilon]^{\binom{n}{2}} \rightarrow [0,\epsilon]^{\binom{n}{2}}$ by $$\varphi(\rho) = \left( \|x_i-x_j\|+ \rho_{ij} - \|\Psi(Z(\rho))_i - \Psi(Z(\rho))_j \|_E \right)_{1 \leq i < j \leq n}.$$ We claim that if $\delta$ is sufficiently small then $\varphi$ is well defined. To see that $\varphi(\rho)_{ij} > 0$, note that $\varphi(\rho)_{ij} \geq \|x_i-x_j\| + \rho_{ij} - \|\Psi(Z(\rho))_i - \Psi(Z(\rho))_j\| = 0$, where we have used that $\|y\| \geq \|y\|_E$ for all $y \in \R^N$.

On the other hand, $\varphi(\rho)_{ij} \leq \|x_i-x_j\| + \rho_{ij} - \frac{1}{1+\delta} \|\Psi(Z(\rho))_i - \Psi(Z(\rho))_j\| = \frac{\delta}{\delta+1} ( \|x_i-x_j\| + \rho_{ij} )$, where we have used that $\|y\| \leq (1+\delta) \|y\|_E$ for all $y \in \R^N$. So if $\delta$ is sufficiently small, then this is less than $\epsilon$.

Since $\varphi$ is a continuous map from a compact convex subset of $\R^{\binom{n}{2}}$ to itself, it follows from Brouwer's fixed point theorem that $\varphi$ has a fixed point $\rho$. Letting $(y_1,\dots,y_n) = \Psi(Z(\rho))$, the map sending $x_i$ to $y_i$ is an isometric embedding of $\{x_1,\dots,x_n\}$ into $E$.
\end{proof}

\begin{rem}
Suppose we had $x_1,\dots,x_n \in \ell_p$ such that the map $\tilde{F}:\underbrace{\ell_p\times \dots \times \ell_p}_{n \text{ times}} \rightarrow U$, $\tilde{F}(y_1,\dots,y_n) = (\|y_i-y_j\|)_{1 \leq i < j \leq n}$, had a continuous right inverse at $\tilde{F}(x_1,\dots,x_n)$. Then an identical argument to the proof of Lemma \ref{lem:final} would show that there is some $\delta > 0$ such that if $d(Y,\ell_p) < 1+\delta$, then $Y$ contains an isometric copy of $\{x_1,\dots,x_n\}$. Since the assumption in Theorem \ref{thm:main} is \emph{weaker} than the Banach space containing an isomorphic copy of $\ell_p$, we had to choose a more technical version of Property $K$ than simply "$\tilde{F}$ has a continuous right inverse at $(x_1,\dots,x_n)$." This stronger assumption also motivated Lemma \ref{lem:ball}.
\end{rem}

We now give the proof of Theorem \ref{thm:main}:

\begin{proof}[Proof of Theorem \ref{thm:main}]
By a combination of Lemmas \ref{lem:ball}, \ref{lem:perturb} and \ref{lem:final}, we see that if an $n$-tuple $(x_1,\dots,x_n)$ in $\ell_p$ has Property $K$, then there is some $N \in \N$ and $\delta > 0$ such that if $Y$ is a Banach space with $d(Y,\ell_p^N) < 1+\delta$, then $\{x_1,\dots,x_n\}$ with the metric inherited from $\ell_p^N$ embeds isometrically into $Y$. By Krivine's Theorem, Theorem \ref{thm:krivine}, any Banach space $X$ satisfying the assumption of the theorem (ie, containing the spaces $\ell_p^n$, $n \in \N$, uniformly), contains a subspace $Y$ with $d(Y,\ell_p^N) < 1+\delta$. Thus $\{x_1,\dots,x_n\}$ with the metric inherited from $\ell_p^N$ embeds isometrically into $X$.

To conclude, we just need to show that the set $A$ of all $n$-tuples that do not have Property $K$ is Haar null. Indeed, the intersection of $A$ with the finite-dimensional space $\ell_p^n \times \dots \times \ell_p^n$ is contained in the complement of $G_n$, which by Theorem \ref{thm:precise} has measure zero. Note also that $A$ is translation-invariant. Thus, by the characterization of Haar null sets stated in Section \ref{sec:ban}, $A$ is Haar null. Since $A$ is closed, it follows that $A$ is nowhere dense.
\end{proof}

\section{Further Remarks and Open Problems}
In this section we give some remarks on the special cases of $\ell_2$, $\ell_\infty$ and $\ell_1$, and pose some open problems.

In the case $\ell_2$, we deduce Theorem \ref{thm:l2} from our results.
\begin{thm}
Every finite affinely independent subset of $\ell_2$ isometrically embeds into every infinite-dimensional Banach space $X$.
\end{thm}
\begin{proof}
First note that every affinely independent set has a linearly independent translate, so without loss of generality, we may reduce to the case of linearly independent sets. Let $e_1,e_2,\dots$ be an orthonormal basis of $\ell_2$. If $\{x_1,\dots, x_n\}$ is a linearly independent subset of $\ell_2$, then there is some isometry $\Theta$ such that $\Theta(x_1) \in $ span$\{e_1\}$, $\Theta(x_2) \in $ span$\{e_1,e_2\}$, etc. Such a $\Theta0$ is constructed by induction and the Gram-Schmidt process applied to the vectors $\{x_1,\dots,x_n\}$. Then a minor variant of Lemma \ref{lem:tri} (in which the coefficient of $e_i$ in $x_i$ is non-zero, but not necessarily one) shows that the $n$-tuple $(x_1,\dots,x_n)$ belongs to $G_n$. Thus $(\Theta x_1,\dots, \Theta x_n)$ (which is isometric to $(x_1,\dots,x_n)$) has Property $K$.

Applying Lemma \ref{lem:final} to $(\Theta x_1,\dots, \Theta x_n)$ we see that there exists some $\delta > 0$ such that whenever $E$ is an $n$-dimensional Banach space with $d(E,\ell_2^n) < 1+\delta$ then $(\Theta x_1,\dots, \Theta x_n)$ embeds isometrically into $E$. By Dvoretzky's theorem, if $X$ is infinite-dimensional, there is a subspace $Z$ of $X$ such that $d(Z,\ell_2^n) < 1+\delta$, and thus $Z$ contains an isometric copy of $(\Theta x_1,\dots, \Theta x_n)$ (which is isometric to $(x_1,\dots,x_n)$).
\end{proof}

In the case of $\ell_\infty$, the proof of Theorem \ref{thm:linf} (given as Theorem 4.3 in \cite{me}) essentially proceeds by directly showing that if $(x_1,\dots,x_n)$ is a concave metric space, then the mapping $\tilde{F}$ is locally open at $(x_1,\dots,x_n)$. This argument does not use differentiation: the norm on $\ell_\infty$ is easy to compute.

In the case of $\ell_1$, the majority of the proofs in this paper simply do not work. In the case $p=1$ the computation of the derivative (Equation \eqref{eq:derivcomp}) yields $\frac{\partial F}{\partial e_l^k} = ($sgn$(x_i^k - x_j^k)_{1 \leq i < j \leq n})$. Thus the function is locally open if the collection forms linearly independent matrices. This is, however, not the case on a large set as it is for the case $1 < p < \infty$. However, if it is true at a point $x = (x_1,\dots,x_n)$ the rest of the proofs presented here work identically.

We now list some open problems. The case $p=2$ was originally raised by Ostrovskii in \cite{ostrov2}, who asked:
\begin{quest}
Let $X$ be an infinite-dimensional Banach space and $A$ a finite subset of $\ell_2$. Then does $A$ isometrically embed into $X$?
\end{quest}

The general question of Ostrovskii, given in \cite{ostrov}, still remains open:
\begin{quest}
Let $X$ be an infinite-dimensional Banach space containing $\ell_p$ isomorphically. Then does every finite subset of $\ell_p$ embed isometrically into $X$?
\end{quest}

The way we approached this question leads to the following natural variant:

\begin{quest}
Let $X$ be an infinite-dimensional Banach space that uniformly contains $\ell_p^n$, $n \in \N$. Then does every finite subset of $\ell_p$ embed isometrically into $X$?
\end{quest}

As detailed in the introduction, there can be no positive results in the cases $p=1$ and $p=\infty$. However, the known partial answers lead to the following open question:
\begin{quest}
Let $p=1$ or $p=\infty$. Which $n$-point subsets of $\ell_p$ embed isometrically into any Banach space $X$ that uniformly contains the spaces $\ell_p^n$, $n \in \N$?
\end{quest}

\section*{Acknowledgements}
I would like to thank my PhD supervisor Andras Zsak for some helpful comments and suggestions. I would like to thank Mikhail Ostrovskii for posing this problem to me, and would like to thank Jack Smith for a helpful conversation.

\end{document}